\documentclass[a4paper]{amsart}

\usepackage[T1]{fontenc}
\usepackage[all]{xy}
\usepackage[latin1]{inputenc}
\usepackage{amssymb}
\usepackage{tensor}
\usepackage{mathrsfs}
\usepackage[english,french]{babel}

\newcommand{\newabstract}[1]{%
  \par\bigskip
  \csname otherlanguage*\endcsname{#1}%
  \csname captions#1\endcsname
  \item[\hskip\labelsep\scshape\abstractname.]
}

\newtheorem*{Maintheorem*}{Main Theorem}
\newtheorem*{theorem*}{Theorem}

\newtheorem{theorem}{Theorem}
\newtheorem{definition}{Definition}
\newtheorem{lemma}{Lemma}
\newtheorem{remark}{Remark}
\newtheorem{proposition}{Proposition}
\newtheorem{corollary}{Corollary}

\newcommand{\bigslant}[2]{{\raisebox{.2em}{$#1$}\left/\raisebox{-.2em}{$#2$}\right.}}

\newcommand{\Ric}{\operatorname{Ric}}
\newcommand{\A}{\mathcal{A}}
\newcommand{\LL}{\mathcal{L}}
\newcommand{\D}{\mathcal{D}}

\newcommand{\Oh}{\mathcal{O}}

\newcommand{\C}{\mathbb{C}}

\newcommand{\NN}{\mathbb{N}}
\newcommand{\T}{\mathcal{T}}

\newcommand{\la}{\alpha}
\newcommand{\lb}{\overline{\beta}}
\newcommand{\ld}{\overline{\delta}}
\newcommand{\lc}{\gamma}

\newcommand{\lt}{\tau}
\newcommand{\ls}{\sigma}

\newcommand{\we}{\wedge}

\newcommand{\Harm}{\mathscr{H}}
\newcommand{\Cka}{\mathscr{C}}
\newcommand{\Ker}{\operatorname{Ker}}
\newcommand{\Ima}{\operatorname{Im}}
\newcommand{\Dom}{\operatorname{Dom}}
\newcommand{\Id}{\operatorname{Id}}

\newcommand{\ovl}{\overline}
\newcommand{\dbar}{\bar \partial}
\newcommand{\dl}{ \partial}

\newcommand\<{\langle}
\renewcommand\>{\rangle}

\newcommand{\DD}{\mathcal{D}}
\newcommand{\X}{\mathcal{X}}

\newcommand{\V}{\mathcal{V}}

\newcommand{\jbar}{\ovl{\jmath}}

\newcommand{\lbar}{\ovl{l}}

\newcommand{\vbar}{\ovl{v}}

\newcommand{\zbar}{\ovl{z}}

\newcommand{\Bbar}{\ovl{B}}

\newcommand{\cbar}{\ovl{\gamma}}

\newcommand{\sbar}{\ovl{s}}

\newcommand{\mi}{\mathrm{i}}
\newcommand{\id}{\operatorname{id}}

\newcommand{\laplace}{\Box_\partial}
\newcommand{\laplacedbar}{\Box_{\dbar}}

\def\ke{Käh\-ler-Ein\-stein}
\def\wp{Weil-Pe\-ters\-son}
\def\ka{Käh\-ler}
\def\wt#1{\widetilde{#1}}
\def\ol#1{\overline{#1}}

\def\pt{{\partial}}

\def\ks{Ko\-dai\-ra-Spen\-cer}

\begin{document}

\title{Positivity of direct images with a Poincar\'e type twist\\}

\author{Philipp Naumann}
\address{Philipp Naumann, Mathematisches Institut,
Universit\"at Bayreuth, 95440 Bayreuth, Germany}
\email{philipp.naumann@uni-bayreuth.de}

\dedicatory{For Tristan}

\thanks{}

\subjclass[2000]{32L10, 32G05, 14Dxx}

\keywords{$L^2$-metric, Curvature of direct image sheaves, Poincar\'e type metric, log forms}

\date{}

\selectlanguage{english}
\begin{abstract}
We consider a holomorphic family $f:\mathcal{X} \to S$ of compact complex manifolds and a line bundle $\LL\to \X$. Given that $\LL^{-1}$ carries a singular hermitian metric that has Poincar\'e type singularities along a relative snc divisor $\DD$, the direct image $f_*(K_{\X/S}\otimes \DD \otimes \LL)$ carries a smooth hermitian metric. In case $\LL$ is relatively positive, we give  an explicit formula for its curvature. The result applies to families of log-canonically polarized pairs. Moreover we show that it improves the general positivity result of Berndtsson-P\u{a}un in a special situation of a big line bundle.

\end{abstract}

\selectlanguage{english}
\maketitle

\section{Introduction}
We consider a proper holomorphic submersion $f: \X \to S$ of complex manifolds and a snc divisor $\DD$ on $\X$ whose restriction $D_s:=\DD|_{X_s}$ to each fiber $X_s=f^{-1}(s)$ is also simple normal crossing. Given a line bundle $\LL \to \X$ that carries a hermitian metric that is smooth on $\X':= \X \setminus \DD$ and whose inverse has Poincar\'e growth near the divisor $\DD$, we study the spaces
of square integrable canonical forms on the open fibers $X'_s:= X_s \setminus D_s$ that have values in $L_s:=\LL|_{X'_s}$. By the work of Zucker \cite{Zu79} and Fujiki \cite{Fu92} we can identify these $L^2$-Dolbeault cohomology groups $H^0_{(2)}(X'_s,K_{X'_s}\otimes L_s)$ with the spaces $H^0(X_s,K_{X_s}\otimes D_s\otimes L_s)$. More globally, we will see that these spaces are the fibers of the coherent sheaf $f_*(K_{\X/S}\otimes \DD \otimes \LL)$ on the base $S$. Under the condition that this sheaf is locally free, the natural $L^2$-metrics on the $L^2$-Dolbeault spaces induce a smooth hermitian metric on the direct image. We give an explicit curvature formula for it in case the hermitian metric on $\LL$ is positive along the fibers $X'_s$. In case $\LL$ is globally positive on $\X':=\X \setminus \DD$, direct image 
$f_*(K_{\X/S}\otimes \DD \otimes \LL)$  is also positive. The result applies to  families of log-canonically polarized pairs where the Poincar\'e type K\"ahler-Einstein metrics induce such a singular hermitian metric on the relative canonical bundle. We give another application to illustrate how our result improves the general positivity theorem from \cite{BP08} in a special situation.

\section{Differential geometric setup and statement of results}
Let $f: \X \to S$ be a proper holomorphic submersion of complex manifolds with connected fibers and $\LL$ a line bundle on $\X$. We assume that $\LL$ has a hermitian metric $h$ that is smooth on the complement of a relative simple normal crossing divisor $\DD=\sum_{i=1}^l{\DD_i}$ on $\X$ with the following asymptotic behaviour along $\DD$: 
\begin{eqnarray} \label{asym}
h^{-1}|_{\X'}= \exp(u)\cdot \frac{h_{\LL^{-1}}^{C^{\infty}}}{\prod_{i=1}^l{||\sigma_i||_i^2\log^2||\sigma_i||_i^2}}
\end{eqnarray}
where the notation is as follows:
\begin{itemize}
\item[$\bullet$] $h_{\LL^{-1}}^{C^{\infty}}$ is a smooth metric on $\LL^{-1}$
\item[$\bullet$]  $||\sigma_i||_i$ is the norm of the canonical section cutting out $\D_i$ w.r.t. a smooth metric s.t. $||\sigma_i||_i<1$
\item[$\bullet$]  $u$ is a function on $\X'$ s.t. $u|_{X'_s} \in \Cka^{k,\alpha}$ for all $s$ and the map $s \mapsto u|_{X'_s}$ is Fr\'echet differentiable
\item[$\bullet$] $\omega_s:= -\mathrm{i} \dl \dbar \log(h)|_{X'_s}$ is a Poincar\'e type K\"ahler metric on each fiber $X'_s$. \\
\end{itemize}

Here we have used the H\"older space of functions 
$\Cka^{k,\alpha}=\Cka^{k,\alpha}(X_{s_0}')$ on an open fiber $X'_{s_0}$ that were introduced in \cite{CY80,Kob84,TY86} and do not depend on the fiber. We refer to this by saying that the inverse metric $h^{-1}$ has \emph{Poincar\'e type singularities} along $\DD$. We write $\DD \overset{i}{\hookrightarrow} \X \overset{f}{\rightarrow} S$ for the family of smooth log pairs $(X_s,D_s)$.

The reason for choosing this asymptotic will become clear when we consider the family of Poincar\'e type \ke\, metrics for a family of log canonically polarized manifolds. We remark that it includes the case where the function $u$ is smooth on $\X$, and it implies that $u$ and its derivatives in base and fiber direction that we will consider are  bounded on $\X'$.
In a local description, the norm squared of a local trivialising section $e_{\LL}$ of $\LL$ near a point $p\in \DD$ is given by
$$
|e_{\LL}|^2_{h}(z,s)=\left( \prod_{i=1}^k{|z^i|^2\log^2(|z^i|^2)}\right) \cdot v(z,s), \quad v \in C^k_{loc}(\X'),
$$ 
where the divisor is given by $\DD=\{ z^1\cdots z^k=0\}$ with respect to local holomorphic coordinates $(z,s)$ around $p$ with $z=z^1,\ldots,z^n$ and $s=s^1,\dots,s^m$ such that $f(z,s)=s$. 
Here $n:= \dim X_s$ and $m:= \dim S$.

The curvature form of the hermitian line bundle $(\LL,h)$ restricted to $\X':=\X\setminus \DD$ is given by
$$
\omega_{\X'} := -\sqrt{-1} \dl\dbar \log h.
$$
This means we view $h$ as a singular hermitian metric on $\LL$  whose curvature current restricted to $\X'$ is given by the smooth form
$\omega_{\X'}$.
Our assumption on $h$ guarantees that each restriction $\omega_{\X'}|_{X_s'}$ is quasi-isometric to the model metric
$$
\sqrt{-1}\left( \sum_{i=1}^k{\frac{dz^i\wedge d\zbar^i}{|z^i|^2\log^2(1/|z^i|^2)}} + \sum_{i=k+1}^n{dz^i\wedge d\zbar^i} \right)
$$  
near the point $p$.

We consider the case where the hermitian line bundle $(\LL,h)|_{\X'}$ is relatively positive, which means that 
$$
\omega_s:=\omega_{\X'}|_{X'_s}
$$
are K\"ahler forms on the open fibers $X'_s:=X_s\setminus D_s$. This implies that $\LL\otimes \DD$ is relatively big and nef. Then one has the notion of the \emph{horizontal lift} $v_i$ of a tangent vector $\dl_i$ on the base $S$ (see section \ref{setup} for the precise definition) and we get a representative of the Kodaira-Spencer class by
$$
A_i:=\dbar(v_i)|_{X'_s}
$$
which is a $\Cka^{k,\la}$-tensor by \cite[Lemma 3]{Sch98} and thus square integrable.
Furthermore, one sets
$$
\varphi_{i\jbar}:=\<v_i,v_j\>_{\omega_{\X'}},
$$
which is called the \emph{geodesic curvature}.  We note that $(\varphi_{i\jbar})_{i\jbar}$ is positive (semi-)definite if and only if $\LL$ is globally (semi-)positive on $\X'$. Again our assumption on $u(z,s)$ guarantees that each $\varphi_{i\jbar}$ is a $\Cka^{k,\la}$ function and thus in particular bounded on each fiber.

Now we turn to the direct image sheaf we want to study.
On a fiber $X_s$, we denote by $\Omega^n_{X_s}(\log D_s)=K_{X_s}\otimes D_s$ the locally free sheaf of germs of logarithmic $n$-forms with log-poles along $D_s$. This sheaf is the restriction of 
$\Omega^n(\log \DD)_{\X/S}=K_{\X/S}\otimes \DD$, the sheaf of relative logarithmic $n$-forms with log-poles along $\DD$, to the fibers $X_s$.

We assume that the dimension of the cohomology groups
$$
H^0(X_s,\Omega^n_{X_s}(\log D_s)(L_s))
$$
is constant on $S$ which in general only holds outside a proper subvariety.
Under this assumption we get the local freeness of the coherent sheaf
$$
f_*(\Omega^n(\log \DD)_{\X/S}(\LL))
$$
whose fibers are canonically isomorphic to the cohomology groups $H^0(X_s,\Omega^n_{X_s}(\log D_s)(L_s))$. By the work of Zucker \cite{Zu79} and Fujiki \cite{Fu92}, we can identify these groups with the $L^2$-Dolbeault cohomology groups $H^0_{(2)}(X'_s,K_{X'_s}\otimes L_s)$. Hence, also the latter spaces form a vector bundle on the base, which we denote by
$$
f_*(K_{\X'/S}\otimes \LL|_{\X'})_{L^2}.
$$
It turns out that this is nothing else than the bundle $f_*(\Omega^n(\log \DD)_{\X/S}(\LL))$.
Now we can represent local sections of $E$ by holomorphic sections of $\Omega^n(\log \DD)_{\X/S}(\LL)$ whose restrictions to the open fibers $X'_s$ are thus $L^2$-integrable and holomorphic
$(n,0)$-forms with values in $L_s$. Let $\{\psi^1,\ldots,\psi^r\}$ be a local frame of the direct image consisting of such sections around a fixed point $s \in S$. We denote by $\{(\dl/\dl s_i)\;|\;i=1,\ldots,m\}$ a  basis of the complex tangent space $T_sS$ of $S$ over $\C$, where $s_i$ are local holomorphic coordinates on $S$. 
The components of the metric tensor for the $L^2$-metric on the direct image are then defined by
$$
H^{\lbar k}(s):=\<\psi^k,\psi^l\> :=\<\psi^k|_{X'_s},\psi^l|_{X'_s}\>(s) := \int_{X'_s}{\psi^k|_{X'_s} \cdot \psi^{\lbar}|_{X'_s}\; dV} = \mi^{n^2} \int_{X'_s}{(\psi^k|_{X'_s} \wedge \psi^{\lbar}|_{X'_s})_h}.
$$
Here we use the notation $\psi^{\lbar}:=\ovl{\psi^l}$ for the sections $\psi^l$ and write $dV =\omega_{X'_s}/n!$. The pointwise inner product $\psi^k \cdot \psi^{\lbar}$ is the one given by $\omega_s$ and 
$h|_{X_s'}$.
In the last equality we used the Hodge-Riemann bilinear relation because the holomorphic $(n,0)$-forms are primitive. Note that by working on the open fibers $X_s'$ the metrics involved are smooth so that we have a harmonic theory for square integrable forms lying in the domain of the Laplacian. 

Let $A_{i\lb}^{\la}(z,s)\dl_{\la}dz^{\lb}=\dbar(v_i)|_{X'_s}$ be the $\dbar$-closed representative of the Kodaira-Spencer class of $\dl_i$ described above. We know that $A_i$ lies in the space of smooth and $L^2$-integrable $(0,1)$-forms $A^{0,1}_{(2)}(X'_s,T_{X'_s})$. Hence these, together with contraction, define a map 
$$
A_{i\lb}^{\la}\dl_{\la}dz^{\lb}\cup \; : H^0(X_s,\Omega^n_{X_s}(\log D_s)(L_s)) \to A^{0,1}_{(2)}(X'_s,\Omega^{n-1}_{X'_s}(L_s)). 
$$
When applying the Laplace operator to $(p,q)$-forms with values in $L|_{X'_s}$ on the fibers $X'_s$, we have
\begin{equation*}
\laplace - \laplacedbar = (n-p-q)\cdot \id
\end{equation*}
due to the definition $\omega_{X'_s}=\omega_{\X'}|_{X'_s}$ and the Bochner-Kodaira-Nakano identity. Thus, we write $\Box=\laplace=\laplacedbar$ in the case $q=n-p$. The main result is
\begin{theorem}
\label{Thm}
\label{mainresult}
Let $\DD \overset{i}{\hookrightarrow} \X \overset{f}{\rightarrow} S$ be a family of smooth log pairs and $(\LL,h)\to \X$ a hermitian line bundle as described above. With the objects just described, the 
$L^2$-metric on $f_*(\Omega^n(\log \DD)_{\X/S}(\LL))$ is smooth and its curvature is given by
\begin{eqnarray*}
R_{i\jbar}^{\lbar k}(s) = &&\int_{X'_s}{\varphi_{i\jbar} \cdot(\psi^k\cdot\psi^{\lbar})\,dV}\\
&+& \int_{X'_s}{(\Box +1)^{-1}(A_i \cup \psi^k)\cdot(A_{\jbar}\cup\psi^{\lbar})\,dV}
\end{eqnarray*}
In particular,  $f_*(K_{\X/S}\otimes \DD \otimes \LL)$ is Nakano (semi-)positive if $\LL$ is (semi-)positive on $\X'$ and positive along the fibers $X'_s$.
\end{theorem}

We remark that in case $\LL$ is (semi-)positive, the direct image $f_*(K_{\X/S}\otimes \DD \otimes \LL)$ is locally free by the Ohsawa-Takegoshi extension theorem.
The result applies to families of log-canonically polarized pairs where the relative canonical bundle $K_{\X'/S}$ plays the role of $\LL$. Here the hermitian metric is induced from the fiberwise Poincar\'e type \ke\; metrics. In this case, we first prove  
\begin{theorem}(= Theorem \ref{posvar})
Let $\DD \overset{i}{\hookrightarrow} \X \overset{f}{\rightarrow} S$ be a family of smooth log-canonically polarized pairs. Then the curvature of the hermitian metric on $K_{\X'/S}$ that is induced by the Poincar\'e type \ke\; metrics on the fibers is semipositive. If the family is effectively parameterized, then $K_{\X'/S}$ is strictly positive. 

This answers a question raised in \cite[Remark 7.1]{Gue16}:

\begin{corollary}(= Corollary \ref{pshvar})
For a family of smooth log-canonically polarized pairs $\DD \overset{i}{\hookrightarrow} \X \overset{f}{\rightarrow} S$ the relative log canonical bundle $K_{\X/S}\otimes\DD$ equipped with the metric induced from the fiberwise K\"ahler-Einstein metrics is nef. If the family is effectively parameterized, $K_{\X/S}\otimes \DD$ is big. 
\end{corollary}

\end{theorem}
By combining both theorems we get 
\begin{corollary}
For a family of log-canonically polarized pairs $\DD \overset{i}{\hookrightarrow} \X \overset{f}{\rightarrow} S$ the direct image sheaf $f_*((K_{\X/S}\otimes\DD)\otimes K_{\X/S})$ is semipositive in the sense of Nakano. In case the family of log pairs is effectively parameterized this direct image is Nakano positive. 
\end{corollary}

In order to implement the method of computation given in \cite{Sch12, Na17}, we have to pass from the compact fibers $X_s$ to its open part $X'_s$ where the metrics in consideration are smooth. This requires to impose the $L^2$ condition on the spaces of forms on $X'_s$. To show that all steps in the computation are still justified, we have to check integrability. This is possible due to knowledge of the asymptotic behaviour of our sections. Another important tool are H\"older spaces with respect to quasi-coordinates. We give the details below.  

\textbf{Acknowledgements.} The author would like to sincerely thank Georg Schumacher for numerous useful discussions about his article \cite{Sch98} and Mihai P\u{a}un for his comments and interest in this article.

\section{Preparations}
\subsection{$L^2$-integrable forms}\label{se:l2int}
We start with a fiberwise consideration.
Let $(X,D=\sum_{i=1}^lD_i)$ be a smooth log pair and set $X'=X\backslash D$. We consider a holomorphic line bundle $L$ on $X$ together with a metric $h$ that is smooth on $X'$ and whose inverse has the asymptotic behaviour from equation (\ref{asym}):
$$
h^{-1}|_{X'}= \exp(u)\cdot \frac{h_{L^{-1}}^{C^{\infty}}}{\prod_{i=1}^l{||\sigma_i||_i^2\log^2||\sigma_i||_i^2}}
$$
where the notation is as follows:
\begin{itemize}
\item[$\bullet$] $h_{L^{-1}}^{C^{\infty}}$ is a smooth metric on $\LL^{-1}$
\item[$\bullet$]  $||\sigma_i||_i$ is the norm of the canonical section cutting out $D_i$ w.r.t. a smooth metric s.t. $||\sigma_i||_i<1$
\item[$\bullet$]  $u$ is a function in $\Cka^{k,\alpha}(X')$
\item[$\bullet$] $\omega_{X'}:= -\mathrm{i} \dl \dbar \log(h)|_{X'}$ is a Poincar\'e type K\"ahler metric\\ 
\end{itemize}
Then we can identify the holomorphic and locally $L^2$-integrable $(n,0)$-forms with values in $L$:
\begin{proposition}\label{pr:fibL2}
We denote by $\Oh_{(2)}(\Omega^n_{X'}(L|_{X'}),h|_{X'})$ the sheaf of holomorphic $L$-valued $n$-forms on $X'$ which are locally $L^2$ on $X$ with respect to $h|_{X'}$. Then
$$
\Oh_{(2)}(\Omega^n_{X'}(L|_{X'}),h|_{X'}) = \Oh(\Omega^n_X(\log D)(L)).
$$
\end{proposition}

\noindent The {\em proof} follows immediately from a Laurent series argument together with the estimates of Poincar\'e type metrics: Sections, which are locally square integrable with respect to a metric with a Poincar\'e type metric extend holomorphically as forms with logarithmic poles to the given snc divisor $D$, (and vice versa).  \qed

Let $\A^{n,q}_{(2)}(L|_{X'})$ denote the sheaf on $X$ of $L$-valued $(n,q)$-fomrs that are locally $L^2$ integrable with respect to $\omega_{X'}$ and $h|_{X'}$, and whose $\overline\partial$-exterior derivatives, taken in the current sense, are also locally $L^2$. We refer to \cite{Zu82} or \cite{BZ98} for more details on the $L^2$-complex of sheaves. 

\begin{proposition}\label{pr:L2res}
The complex $(\A^{n,\bullet}_{(2)}(L|_{X'}), \overline\partial)$ of sheaves on $X$ is a fine resolution of $\Oh_{(2)}(\Omega^n_{X'}(L|_{X'}),h|_{X'})$. Thus the $L^2$-Dolbeault cohomology group $H^0_{(2)}(X', \Omega^n_{X'}(L|_{X'}))$ can be identified with $ H^0(X, \Omega^n(\log D)(L))$ which is of finite dimension. 
\end{proposition}
\begin{proof}
We decompose the vector bundle locally as a sum of line bundles and apply \cite[Prop. 2.1]{Fu92}, cf. also \cite[p.870]{Fu92}. This shows that we have a resolution. The fact that we get indeed a \emph{fine} resolution is not automatic, cf. \cite[p.175]{Zu82} because we need cut-off functions with bounded differential. But this holds in the context of Poincar\'e geometry, see \cite{Zu79} where it was successfully employed in the one-dimenional case. 
\end{proof}

\subsection{Quasi-coordinates and Hölder spaces}
\label{sectionHoelderspaces}
We recall from \cite{CY80} that a \emph{quasi-coordiante map} is a holomorphic map from an open set $V \subset \C$ into $X'$ if it is of maximal rank everywhere in $V$. In this case, $V$ together with the Euclidean coordinates of $\C^n$ is called a \emph{local quasi-coordinate} of $X'$. According to \cite{CY80,Kob84,TY86}, we have the following
\begin{proposition}
\label{quasicoor}
There exists a family $\V=\{\left(V;v^1,\ldots,v_n\right)\}$ of local quasi-coordinates of $X'$ with the following properties:
\begin{itemize}
\item[(i)] $X'$ is covered by the images of the quasi-coordinates in $\V$.
\item[(ii)] The complement of some open neighbourhood of the divisor $D$ in $X$ is covered by the images of finitely many of the quasi-coordinates in $\V$ which are local coordinates in the usual sense.
\item[(iii)] For each $(V;v^1,\ldots,v^n) \in \V, v \subset \C^n$ contains an open ball of radius $\frac{1}{2}$.
\item[(iv)] There are constants $c>0$ and $A_k>0, k+0,1,\ldots,$ such that for every $(V;v^1,\ldots,v^n) \in \V,$ the following inequality hold:
\begin{itemize}
 \item[$\bullet$] We have 
 $$
 \frac{1}{c} (\delta_{i\jbar}) < (g_{i\jbar}) < c (\delta_{i\jbar})
 $$
 as matrices in the sense of positive definiteness.
 \item[$\bullet$] For any multiindices $I=(i_1,\ldots,i_p)$ and $J=(j_1,\ldots,j_q)$ of order $|I|=i_1+\ldots + i_p$ respectively $|J|=j_1+\ldots + j_q$ we have
 $$
 \left| \frac{\dl^{|I|+|J|} g_{i\jbar}}{\dl v^I\dl \ovl{v}^J} \right| < A_{|I|+|J|},
 $$
 where $\dl v^{I}=(\dl v^1)^{i_1} \cdots (\dl v^p)^{i_p}$ and $\dl \ovl{v}^J=(\dl \ovl{v}^1)^{j_1}\cdots (\dl \ovl{v}^q)^{j_q}$.
\end{itemize}
\end{itemize}
\end{proposition}
A complete Kähler manifold $(X',g)$ which admits a family $\V$ of local quasi coordinates satisfying the conditions of the proposition is called of \emph{bounded geometry} (of order $\infty$).

Although the coordinate system from Proposition \ref{quasicoor} is not a coordinate system in the ordinary sense because of the covering map involved, it makes sense to talk about the components of a tensor field on $X'$ (or a neighbourhood $U(p)=(\Delta^*)^k \times \Delta^{n-k}$) with respect to these "coordinates" $v^i$ by first lifting it to a tensor field $\Delta^n$. The behaviour of the tensor on $U(p)$ can thus be examined by looking at the lifted function in a neighbourhood of $(1,\ldots,1,*,\ldots,*)$ in $\Delta$.

We mention here the transition of tensors from a local coordinate function $z^i$, where a component $D_i$ is given by $\{z^i=0\}$ to a quasi-coordinate $v^i$, $|v^i|< R <1$, which is given by
\begin{equation}\label{eq:quasi}
  \frac{dz^i}{z^i\log|z^i|^2}= \frac{|v^i-1|^2}{(v^i-1)^2}\frac{dv^i}{1-|v^i|^2} \sim dv^i,
\end{equation}
because $v$ is bounded away from $1$. The respective equation for $\partial/\partial z^1$ reads as
\begin{equation}\label{eq:quasi_vf}
 z^i \log |z^i|^2  \frac{\dl}{\dl z^i} \sim \frac{\dl}{\dl v^i}. 
\end{equation}
This transformation rule is used for deriving estimates from H\"older regular functions and tensors that we will define now.
The Hölder spaces $\Cka^{k,\la}(X')$ are defined in terms of quasi-coordinates for sufficiently large values of $k$ and $0<\la<1$. The Hölder norms are computed in terms of the infinite number of quasi-coordinate systems. Following \cite{CY80, Kob84, TY86} we define:

\begin{definition}
\label{Hoelderspaces}
Let $k \in \NN_0$ and $\la \in (0,1)$ and denote by $C^k(X')$ the space of $k$-times differentiable functions $u: X' \to \C$.
For $u \in C^k(X')$ let   
$$
\Vert u \rVert_{k,\la} = \sup_{(V,v^1,\ldots,v^n) \in \V} \left( \sup _{z \in V} \sum_{|I|+|J|\leq k}{|\dl_v^J\ovl{\dl}_v^J u(z)|} 
+ \sup_{z,z' \in V} \sum_{|I|+|J|=k} \frac{|\dl_v^I\ovl{\dl}_v^J u(z)-\dl_v^I\ovl{\dl}_v^J u(z')|}{|z-z'|^{\la}} \right)
$$
be the $\Cka^{k,\la}$-norm of $u$, where $\dl_v^I\ovl{\dl}_v^J=\frac{\dl^{|I|+|J|}}{\dl v^I\dl\ovl{v}^J}$.
Then let
$$
\Cka^{k,\la}:=\Cka^{k,\la}(X') := \{u \in C^k(X')\;:\; \Vert u \Vert_{k,\la} <\infty\}.
$$
be the function space of $\Cka^{k,\la}$ functions on $X'$ with respect to $\V$.
\end{definition} 
 
In a similar way we can define $\Cka^{k,\la}$-tensors by pulling back via the quasi-coordinate maps.
Exterior derivatives and covariant derivatives of $\Cka^{k,\la}$-tensors are of the same type (with $k$ being replaced by $k-1$). The arguments from \cite{Kob84} using Hölder spaces $\Cka^{k,\la}$ (with respect to quasi-coordinates) are valid for all large fixed numbers $k\geq k_0$, where $k_0$ denotes some minimal degree. During the course of computations we will have to take derivatives, products and contractions of such tensors, arriving at $\Cka^{k,\la}$-tensors for some lower value of $k$. In each of these steps we will have to increase the lower bound $k_0$. Only finitely many such steps will be necessary. We will increase $k_0$ tacitly.

Let $\omega_{X'}$ be the complete Poincar\'e type K\"ahler form from the previous section.

\begin{lemma}
  Any $\Cka^{k,\la}$-tensor on $X'$ is globally square-integrable.
\end{lemma}
The {\em proof} follows immediately from the given uniform bounds of such tensors in terms of the above quasi-coordinate systems and the metric $\omega_{X'}$ with respect to the transition equations of the type \eqref{eq:quasi}: Take the pointwise norm of such a tensor with respect to the given metric. We consider the resulting function as a bounded $\Cka^{k,\la}$-tensor (for some value of $k$), which means that it is uniformly  bounded in terms of quasi-coordinate systems (or any other coordinate systems). Finally we use the boundedness of the volume of $X'$. \qed

\subsection{Hodge theory on the open fibers}
In this section, we summarize the Hodge theory on the complete, non-compact K\"ahler manifold $(X',\omega_{X'})$ for which we refer to the book of Marinescu and Ma \cite[Chapter 3]{MM07}. We consider the holomorphic vector bundle 
$E=K_{X'}\otimes L|_{X'}$ equipped with the hermitian metric $h|_{X'}$. 

For the operator $\dbar$ defined on the smooth, compactly supported forms $A_0^{0,q}(X',E) \subset L^{0,q}_{(2)}(X',E)$ we consider its formal adjoint $\dbar^*$. For $s_1 \in L_{(2)}^{0,q}(X',E)$, we can calculate $\dbar s_1$ in the sense of currents: $\dbar s_1$ is the currents defined by
$$
\<\dbar s_1,s_2\> = \<s_1,\dbar^*s_2\> \quad \mbox{ for } s_2 \in A_0^{(0,q+1)}(X',E).
$$ 
Then we have
\begin{lemma}\cite[Lemma 3.1.1]{MM07}
The operator $\dbar_{\max}$ defined by
\begin{eqnarray*}
&&\Dom (\dbar_{\max}) = \{s \in L_{(2)}^{0,q}(X',E)\; : \; \dbar s \in L_{(2)}^{0,q+1}(X',E)\}\\
&& \dbar_{\max} s = \dbar s \quad \mbox{ for } s \in \Dom(\dbar_{\max})
\end{eqnarray*}
is a densely defined, closed extension, called the \emph{maximal extension} of $\dbar.$
\end{lemma} 
In the sequel, we work with the maximal extension and simply write $\dbar = \dbar_{\max}$. The Hilbert space adjoint of $\dbar_{\max}$ is denoted by $\dbar^*_H$. We note that $\dbar^*_H \subset (\dbar^*)_H$ and $\dbar_{\max} = (\dbar^*_H)^*_H$. From now on we work with  $\dbar^*_H$ and just write $\dbar^*$. 

The Laplacian $\Box=\dbar \dbar^*+\dbar^*\dbar$ is a densely defined, positive operator, so we one can consider its Friedrichs extension. But in the context of $L^2$ cohomology, its useful to consider another extension: 
\begin{proposition}\cite[Proposition 3.1.2]{MM07}
The operator defined by
\begin{eqnarray*}
&& \Dom (\Box) = \{s \in \Dom(\dbar) \cap \Dom(\dbar^*)\; : \; \dbar s \in \Dom(\dbar^*), \; \dbar^* s \in \Dom(\dbar)\},\\
&& \Box s = \dbar^* \dbar s + \dbar \dbar^* s \quad \mbox{ for } s \in \Dom(\Box)
\end{eqnarray*}
is a positive self-adjoint extension of the Laplacian, called the Gaffney extension.  
\end{proposition}  
We note that this result relies on Gaffney's generalisation of Stokes Theorem that will we use tacitly during our computation:
\begin{proposition}\cite{Ga54}
\label{GaffneyStokes}
Let $(M,g)$ be an orientable complete Riemannian manifold of real dimension $2n$ whose Riemann tensor is of class $C^2$. Let $\gamma$ be a $(2n-1)$-form on $M$ of class $C^1$ such that both 
$\gamma$ and $d \gamma$ are in $L^1$. Then 
$$
\int_M{d\gamma} =0.
$$
\end{proposition}

We define the space of harmonic forms $\Harm^{0,q}(X',E) \subset L_{(2)}^{0,q}(X',E)$ by
$$
\Harm^{0,q}(X',E) := \Ker (\Box) = \{s \in \Dom(\Box)\;:\: \Box s =0\}.
$$
We see that
$$
\Harm^{0,q}(X',E) = \Ker(\dbar) \cap \Ker(\dbar^*).
$$
The $q$-th $L^2$ Dolbeault cohomology is defined by
$$
H^{0,q}_{(2)}(X',E):= \bigslant{\Ker(\dbar) \cap L_{(2)}^{0,q}(X',E)}{\Ima(\dbar) \cap  L_{(2)}^{0,q}(X',E)}
$$
which is of course the same as the group $H^q_{(2)}(X',K_{X'}^p(L|_{X'}))$ from Proposition \ref{pr:L2res}, because the $L^2$-cohomology can be also computed by smooth forms.
Form this we also get
\begin{proposition}
The $L^2$-Dolbeault cohomology $H^{0,q}_{(2)}(X',E) $ is finite-dimensional and we have
$$
H^{0,q}_{(2)}(X',E) \cong \Harm^{0,q}(X',E). 
$$
Moreover, the images of $\dbar$ and $\dbar^*$ are closed in $L_{(2)}(X',E)$ and thus there exists a constant $C>0$ such that
\begin{eqnarray}
\label{fundest0}
||s||^2_{L^2} \leq C \left(  ||\dbar s||^2_{L^2} + ||\dbar^* s||^2_{L^2}\right)
\end{eqnarray}
for all $s \in \Dom(\dbar) \cap \Dom{\dbar^*}\cap L^{0,q}_{(2)}(X',E), \,s \perp \Harm^{0,q}(X',E)$. 
\end{proposition} 
\begin{proof}
By Proposition \ref{pr:L2res}, we know that we can identify the $L^2$-Dolbeault cohomology with the sheaf cohomology of the locally free sheaf $\Omega^p(\log D) \otimes K_X$ on the compactification 
$X$, hence it must be finite-dimensional. This implies immediately that the image of $\dbar$ (and hence also of $\dbar^*$) are closed (\cite[Proposition 3.5]{BZ98}) and the isomorphism with the space of harmonic forms (\cite[Corollary 3.6]{BZ98}). The stated estimate for the $L^2$-norm  for a section $s$ that is orthogonal to the space of harmonic forms follows from the closedness of $\Ima(\dbar)$ and $\Ima(\dbar^*)$ (\cite[Prop. 3.1.6]{MM07}).
\end{proof}
\begin{corollary}
\label{GreenOp}
We have the strong Hodge decomposition which is orthogonal:
\begin{eqnarray*}
L^{0,q}_{(2)}(X',E) &=& \Harm^{0,q}(X',E) \oplus \Ima(\Box) = \Harm^{0,q}(X',E) \oplus \Ima(\dbar\dbar^*) \oplus \Ima(\dbar^*\dbar),\\
\Ker(\dbar) \cap L^{0,q}_{(2)}(X',E)  &=&  \Harm^{0,q}(X',E) \oplus \left( \Ima(\dbar) \cap L^{0,q}_{(2)}(X',E) \right).
\end{eqnarray*}
Moreover, there exists a bounded operator $G$ on $L^{0,q}_{(2)}(X',E)$, called the \emph{Green operator}, such that
$$
\Box G = G \Box = \Id -H, \quad HG=GH=0,
$$
where $H$ is the orthogonal projection form $L^{0,q}_{(2)}(X',E)$ onto $\Harm^{0,q}(X',E)$.
\end{corollary}
\begin{proof}
See \cite[Theorem 3.1.8]{MM07}.
\end{proof}
\begin{remark}
By the usual elliptic regularity theory, we get that harmonic sections are in fact smooth, $G$ maps smooths forms to smooth forms so that we can cut down the Hodge decomposition to the space of smooth $L^2$ sections:
$$
A^{0,q}_{(2)}(X',E) = \Harm^{0,q}(X',E) \oplus \Ima(\Box).
$$
\end{remark}

\subsection{Families of logarithmic pairs}
\label{famlogpairs}
Let $(\X,\DD)$ be a smooth log pair, i.e. $\X$ a complex manifold and $\DD \subset \X$ a reduced snc divisor. The boundary divisor $\DD$ is written as a sum of its irreducible components $\DD = \D_1+\ldots + \DD_k$. If $I \subset \{1,\ldots,k\}$ is any non-empty subset, we consider the intersection $\DD_I:=\cap_{i \in I}\DD_i$.
\begin{definition}\cite[Def. 3.4]{Ke13}
For a smooth log pair $(\X,\DD)$ and a proper holomphic submersion $f: \X \to S$ we say that $\DD$ is relatively snc or that $f$ is a snc morphism if for any set $I$ with $D_I \neq \emptyset$ all the restricted morphisms $f|_{\DD_I}\to S$ are also smooth of relative dimension $\dim \X - \dim S-|I|$.
\end{definition} 
If $s \in S$ is any point, set $X_s=f^{-1}(s)$ and $D_s:= \DD \cap X_s$. Then $X_s$ is smooth and $(X_s,D_s)$ is a snc pair. Moreover, the number of irreducible components of $D_s$ is also $k$ which is the number of irreducible components of $\DD$. This excludes phenomena like the deformation of the smooth hyperbola $\{wz=s \neq 0\}$ into the snc divisor $\{zw=0\}$ that has two components. The reason for this will became clear later on. Moreover, the definition implies 
\begin{lemma}
For a smooth log pair $(\X,\DD)$ and a smooth snc morphism $f:\X \to S$ we can find (after shrinking $S$ to a contractible neighbourhood) a differentiable trivialisation
$\Phi: \X \xrightarrow{\sim} X \times S$ such that the restriction $\Phi|_{\DD}:\DD \xrightarrow{\sim}\ D \times S$ is also a smooth trivialisation. 
\end{lemma}
The lemma says that  we can trivialise the pair $(X_s,D_s)$ in a differentiable way. By restriction to $\X':=\X\setminus \DD$ we find a smooth trivialisation of $\X'\cong X'\times S$.

Given a smooth log pair $(X,D)$ and a point $x$, there exists an open neighbourhood $U=U(x)$ and holomorphic coordinates $z_1, \dots,z_n$ such that $D\cap U=\{z_1\cdots z_r\}$ for some $0\leq r \leq n$. The following is the relative analogue of this fact:

\begin{lemma}\cite[Lemma 3.7]{Ke13}
Let $(\X,\DD)$ be a smooth snc pair and $f: \X \to S$ an snc morphism. For any point $x \in \X$, there exist open neighbourhoods $V=V(f(x))\subset S$  and $U=U(x) \subset f^{-1}(V) \subset \X$ and holomorphic coordinates $z_1,\dots, z_n, z_{n+1},\dots, z_{n+m}$ around $x$ and $s_1,\dots,s_m$ around $f(x)$ and a number $0\leq r\leq n$ such that the following holds:
\begin{itemize}
\item[(i)] We have $z_{n+i}=s_i \circ f$ for all indices $1\le i\leq m$
\item[(ii)] $\DD \cap \X=\{z_{1}\cdots z_r=0\}$.
\end{itemize}
\end{lemma}
In the following, we will make use of these \emph{adapted} coordinates tacitly. In the applications, we also consider the following families:
\begin{definition}\label{deffamlogpairs}
A holomorphic family of log canonically polarized pairs consists of a smooth snc pair $(\X,\DD)$ and a snc morphism $f:\X \to S$ sucht that $K_{X_s}+D_s$ is ample for each $s \in S$. 
\end{definition}
In particular, our families are \emph{logarithmic deformations} in the sense of \cite{Ka78} and generalise the deformations considered in \cite{Sch98} to the case of a singular snc divisors. 
\begin{remark}
By the $C^{\infty}$ trivialisation $\X'\cong X'\times S$ and the fact that nearby fibers $(X_s,g_s)$ and bundles $E'_s=K_{X'_s}\otimes L|_{X'_s}$ are quasi-isometric imply that the $L^2$-spaces 
$L^{0,q}_{(2)}(X'_s,E'_s)$ and the domain of the Laplacians $\Dom(\Box_s)$ do not depend on the fiber $X'_s.$ The same holds for the H\"older spaces $\Cka^{k,\la}(X'_s)$.
\end{remark}

\subsection{$L^2$-integrable \ks\ forms}\label{se:famke}

We consider a holomorphic family of smooth log pairs  $\DD \overset{i}{\hookrightarrow} \X \overset{f}{\rightarrow} S$ and a hermitian holomorphic line bundle $(\LL,h)$ on $\X$ whose inverse metric has Poincar\'e type singularities along $\DD$ as already described in the introduction. Moreover, we recall that we have defined on $\X'$ the global $(1,1)$-form 
$$
\omega_{\X'} = -\mi\partial\overline\partial (\log h)= \mi \dl \dbar \log(h^{C^{\infty}}_{\LL^{-1}})-\sum_{i=1}^l{\mi\dl \dbar\log(||\sigma_i||^2 \log^2||\sigma_i||^2})+\mi\dl \dbar u,
$$
whose restrictions to the open fibers $X'_s$ give a smooth family of Poincar\'e type K\"ahler forms $\omega_s = \omega_{\X'}|_{X'_s}$.

By using the local description
$$
\omega_{\X'}=\sqrt{-1} \left( g_{\la\lb} dz^\la \wedge dz^{\lb} + g_{i\lb} ds^i\wedge dz^{\lb}+ g_{\la \overline{\jmath}} dz^{\la} \wedge ds^{\overline\jmath} + g_{i\overline \jmath} ds^i\wedge ds^{\overline\jmath} \right)
$$
with respect to the local coordinates $(z,s)$ we can say that the restrictions of $ \omega_{\X'}$ as well as restrictions of contractions depend in a $C^\infty$ way upon the parameter:

\begin{lemma}\label{le:contr}
  The restrictions
  \begin{eqnarray*}
    \omega_{\X'}\llcorner (\partial/\partial s^i)|X'_s &=& g_{i\overline\beta}dz^{\overline\beta}|X'_s\\
     \omega_{\X'}\llcorner (\partial/\partial s^{\overline \jmath})|X'_s &=& g_{\alpha\overline \jmath}dz^\alpha|X'_s\\
     \omega_\X'\llcorner (\partial/\partial s^i\wedge ds^{\overline\jmath})|X'_s&=& g_{i\overline \jmath}|X'_s
  \end{eqnarray*}
  are $\Cka^{k,\la}$-tensors that depend in a $C^\infty$ way upon the parameter $s\in S$. In particular, they are smooth and $L^2$ integrable tensors $(\A_{(2)}$-tensors for short). The analogous statement holds, if $\omega_{\X'}$ is replaced by $\widetilde \omega_{\X'}$, where the first two tensors do not change.
\end{lemma}
\begin{proof}
The only critical term appearing in the expression of $\omega_{\X'}$ is the one coming from $u$. The purely vertical components are dealt with the fact that $u|_{X'_s}$ lies in $\in \Cka^{k,\la}$.
But now by assumption the map $s \mapsto u(z,s) \in \Cka^{k,\la}$ is Fr\'echet differentiable, so that the (locally defined) partial derivatives in base direction of the function $u$ are again of class $\Cka^{k,\la}$. 
\end{proof}
{\em Horizontal lifts} of tangent vectors from $S$ to the total space $\X'$ are defined as perpendicular to the fibers of $\X' \to S$ with respect to $\omega_{\X'}$. Here we only need the property that the restrictions of $\omega_{\X'}$ to the fibers are positive definite.

We denote by $(g^{\overline\beta\alpha})$ the inverses of the metric tensors $g_{\alpha\overline\beta}$ of $\omega_s$ on the fibers $X'_s$. Covariant derivatives are always taken with respect to these metrics. We will use both the semicolon notation and the $\nabla$-notation. We continue using Greek indices for tensors in fiber direction, and $\partial_\alpha = \partial/\partial z^\alpha$.

The horizontal lift $v_i$ of a tangent vector $\dl_i=\pt/\pt s^i$ is a differentiable lift of $\dl_i$ to $\X'$ which is orthogonal to the fibers with respect to the sesquilinear form $\omega_{\X'}$:
$$
\<v_i, \dl_{\la}\>_{\omega_{\X'}} =0 \quad \mbox{for all } \la=1\ldots n.
$$
This is well defined, since the form $\omega_{\X'}$ is positive when restricted to the fibers. In terms of the coefficients of
$\omega_{\X}$, it is given by
$$
v_i=\dl_i + a_i^{\la}\dl_{\la},
$$
where
$$
a_i^{\la} = - g^{\lb\la}g_{i\lb}.
$$
\begin{lemma}\label{le:AiCkl}
  Let $A_i = \overline\partial v_i|X'_s$, i.e.\
  $$
  A_i = A^\alpha_{i\overline\beta}\partial_\alpha dz^{\overline\beta}.
 $$
  Then $A_i$ is of class $\Cka^{k,\la}$ and satisfies $\overline\partial A_i=0$. In particular it lies in $A^{0,1}_{(2)}(X'_s,T_{X'_s})$.
\end{lemma}
The proof is the same as in \cite[Lemma 3]{Sch98} \qed

The \ks\ map for a family  $\DD \overset{i}{\hookrightarrow} \X \overset{f}{\rightarrow} S$
$$
\rho_s:T_sS \to H^1_{(2)}(X'_s, T_{X'_s})
$$
was already defined in \cite{Sch98}. Analogous results like those of Section~\ref{se:l2int} hold for the sheaf of holomorphic vector fields:

{\em The sequence $(\A^{0,\bullet}_{(2)}(\T_{X'_s}),\ol\pt)$ is a fine resolution of the sheaf of $L^2$ holomorphic vector fields $\Oh_{(2)}(T_{X'_s})$, which is isomorphic to 
$T_{X_s}(-\log D_s):=(\Omega^1_{X_s}(\log D_s))^\vee$. }

In particular 
$$
H^1_{(2)}(X'_s, T_{X'_s}) \cong H^1(X_s,T_{X_s}(-\log D_s))
$$
and the above explicit construction yields a $L^2$-Dolbeault representative $A^\alpha_{i\ol\beta}\pt_\alpha dz^{\ol\beta}$ of $\rho_s(\pt/\pt s^i)$, where the \ks\ map is taken as map
$$
\rho_s:T_sS \to H^1_{(2)}(X'_s,T_{X'_s}).
$$
We will need also the following fact that follows from the definition.
\begin{lemma} \label{symm}
Let $A_i^{\la \sigma}=g^{\lb \sigma}A_{i\lb}^{\la}$. Then
$$ 
A_i^{\la \sigma}=A_i^{\sigma \la}.
$$
\end{lemma}
Finally, the pointwise inner product with respect to $\omega_{\X'}$ of two horizontal lifts
$$
\varphi_{i\jbar} := \<v_i,v_j\>_{\omega_{\X'}},
$$
called the geodesic curvature, has an expression
$$
\varphi_{i\jbar} = g_{i\jbar} - g_{\la\jbar}g_{i\lb}g^{\la\lb}
$$
so that we conclude from Lemma \ref{le:contr} that this function lies in $\Cka^{k,\la}$.

\subsection{Bundles of $L^2$-integrable forms}
Now we consider the $L^2$ condition on the total space $\X$. Because the form $\omega_{\X'}$ is fiberwise positive, for any $s_0\in S$ after replacing $S$ by a neighborhood (contractible and Stein) there exists a positive hermitian form $\omega_S$ such that $\widetilde{\omega}'_\X = \omega_{\X'}+f^*\omega_S$ is a positive form on $\X'$.

The previous arguments imply that the statements of Proposition~\ref{pr:fibL2} and Proposition~\ref{pr:L2res} hold for the total spaces $(\X',\widetilde\omega_{\X'})$. The sheaves of locally $L^2$ sections $\Oh_{(2)}(\Omega^n_{\X'}(\LL|_{\X'}), \widetilde\omega_{\X'},h|_{\X'})$ on the whole total space $\X$ are equal to $\Omega^n_\X(\log \DD)(\LL)$, and an analogous fine resolution in terms of square-integrable forms exists. However, the sheaves $\Omega^n_\X(\log \D)$ of log $n$-forms on the total space $\X$ are not suitable for our methods.

Instead, we will need the coherent sheaf $\Omega^n(\log \D)_{\X/S}(\LL)$, and the sheaves
$$
\Oh_{(2)}(\Omega^n_{\X'/S}(\LL|_{\X'}), \wt\omega_{\X'},h|_{\X'})
$$
of relative $\LL$-valued holomorphic $n$-forms that are square-integrable with respect to the \ka\ form $\wt\omega_{\X'}= \omega_{\X'}+ f^*\omega_S$ and the hermitian metric $h|_{\X'}$, where integrability does not depend upon the choice of a hermitian form $\omega_S$.

Let
$$
\A^{0,q}_{(2)}(\Omega^n_{\X'/S}(\LL),\wt\omega_{\X'},h|_{\X'})
$$
denote the sheaf of {\em $(0,q)$-currents on the total space $\X'$ with values in the coherent sheaf\/ $\Omega^n_{\X'/S}(\LL)$ that are square-integrable along with their exterior $\ol\pt$-derivatives.}

\begin{proposition}
\begin{itemize}
  \item[(i)] 
  $$
  \Oh_{(2)}(\Omega^n_{\X'/S}(\LL|_{\X'}), \wt\omega_{\X'},h|_{\X'}) \simeq \Omega^n(\log \D)_{\X/S}(\LL).
  $$
\item[(ii)]
  The complex
  $$
  (\A^{0,\bullet}_{(2)}(\Omega^n_{\X'/S}(\LL),\wt\omega_{\X'},h|_{\X'}),\ol\pt)
  $$
  is a fine resolution of $\Oh_{(2)}(\Omega^n_{\X'/S}(\LL), \wt\omega_{\X'},h|_{\X'})$.
\end{itemize}
\end{proposition}
\begin{proof}
To simplify notation, we assume that $\dim S=1$. We note that
$$
\widetilde\omega_{\X'}^{n+1} = \sqrt{-1}\omega_{\X'/S}^n\;\varphi\; ds\wedge d\sbar + \omega_{\X'/S}^n \wedge f^*\omega_S
$$
and the function $\varphi$ is (locally w.r.t. $S$ uniformly) bounded. For an open set $U \subset \X$, we have thus for the $L^2$-norm of a section $u \in \Omega^n_{\X'/S}(\LL)$ that
$$
\int_U{|u|^2_{\omega_{\X'/S}, h_{\LL}}\;\widetilde\omega_{\X'}^{n+1}} \sim \int_{f(U)}\left(\int_{X_s\cap U}{(u\wedge \ovl{u})_h}\right)\omega_S
$$
Now the first statement follows from Fubini's Theorem, the fact that nearby fibers and bundles are quasi-isometric and Proposition ~\ref{pr:fibL2}. For the second statement, we decompose the vector bundle 
$\Omega^n_{\X'/S}(\LL|_{\X'})$ locally as a sum of line bundles and apply \cite[Prop. 2.1]{Fu92}, cf. also \cite[p.870]{Fu92}. This shows that we have a resolution. 
By Proposition~\ref{pr:L2res} fiberwise we find cut-off functions with bounded differential. Because in the differentiable sense we have a product situation, this holds on the total space, too. Hence, the sheaves of locally $L^2$ integrable smooth sections are again fine.
\end{proof}
\begin{corollary}\label{cor:repr}
The (local) holomorphic sections of $f_*(\Omega^n(\log \D)_{\X/S}(\LL))$ are given by holomorphic sections of $\Omega^n(\log \D)_{\X/S}(\LL)$ on the total space which are precisely the holomorphic sections of 
$\Omega^n_{\X'/S}(\LL|_{\X'})$ that are $L^2$-integrable along the fibers. In particular, their restrictions to the open fibers yield holomorphic and $L^2$-integrable $(n,0)$-forms with values in $L$.
\end{corollary}

\subsection{Fiber integrals and Lie derivatives}
Given our family $f: \X' \to S$ of open complex manifolds $X'_s$ of dimension $n$ and a relative $(n,n)$-form $\eta$ on $\X'$ that is smooth there and integrable along the fibers,
the fiber integral
$$
\int_{\X'/S}{\eta}
$$
gives a function on the base $S$ (see \cite[Sect. 2.1]{Sch12} and \cite[Ch. VII]{GHV72} for the general definition of fiber integrals). In our case, the components $H^{\lbar k}$ of the metric tensor on the direct image are defined by such fiber integrals, where $\eta$ is given by the inner product/wedge product of the sections $\psi^k$:
$$
\eta \,(s) = \psi^k|_{X'_s} \cdot \psi^{\lbar}|_{X'_s} \;dV = \mi^{n^2} \; (\psi^k|_{X'_s} \wedge \psi^{\lbar}|_{X'_s})_h 
$$
We want to show that these fiber integrals give  smooth functions on the base so that we get indeed a smooth hermitian metric on the direct image we consider.
 Thus, if $s^1,\ldots,s^m$ are local holomorphic coordinates on the base, we need to compute the derivatives
$$
\frac{\dl}{\dl s^k}\int_{X_s}\eta \quad  \mbox{for }1\leq i \leq r \quad\mbox{ and } \quad \frac{\dl}{\dl s^{\lbar}}\int_{X_s}\eta, \quad \mbox{for }
1\leq l \leq r.
$$ 
This can be done by using Lie derivatives:
\begin{lemma}
\label{lemma 1}
For $1\leq k \leq m$, let $v_k$ be the horizontal lift of $\dl/\dl s^k$. We write $\dl/\dl s^{\lbar}$ for 
$\dl/\dl \ovl{s^l}$ and $v_{\lbar}$ for $\ovl{v_l}$. Then
$$
\frac{\dl}{\dl s^k}\int_{X'_s}\eta = \int_{X'_s}L_{v_k}(\eta) \quad \mbox{and} \quad \frac{\dl}{\dl s^{\lbar}}\int_{X'_s}\eta = \int_{X'_s}L_{v_{\lbar}}(\eta),
$$
where $L_{v_k}$ and $L_{v_{\lbar}}$ denotes the Lie derivative in the direction of $v_k$ and $v_{\lbar}$ respectively.
\end{lemma}
\begin{proof}
The statement is well-known when the fibers are compact \cite[Lemma 1]{Na17}.
We only have to show that $L_{v_k}(\eta)$ and $L_{v_{\lbar}}(\eta)$ are square integrable. Then the statement follows from the dominated convergence theorem. We only present it for the first.
We are using local holomorphic coordinates $z^1,\ldots,z^n$ near a point $p \in D_s$ such that $D_s=\{z^1\cdots z^k=0\}$.
Because $\psi^k,\psi^l \in H^0(\X,\Omega^n(\log \D)_{\X/S}(\LL))$ and the assumption on the metric $h$ we have that
$$
 \eta_{1\ldots n \ovl{1} \ldots \ovl{n}} = O\left(\prod_{i=1}^k{\log^2(|z^i|^2)}\right).
$$
where
$$
\eta = \eta_{1\ldots n \ovl{1} \ldots \ovl{n}} \; dz^1 \wedge \ldots \wedge dz^n \wedge dz^{\ovl{1}} \wedge \dots \wedge dz^{\ovl{n}}.
$$ 
Now $v_k = \dl/\dl_k + a_k^{\la}\,\dl/\dl z^{\la}$ so that
$$
L_{v_k}\eta  = \left(\eta_{1\ldots n, \ovl{1}, \dots \ovl{n};k} + \sum_{\la=1}^na_k^{\la}\eta_{1\ldots n \ovl{1} \ldots \ovl{n};\la} + \sum_{\la=1}^n{a_{k;\la}^{\la}
\eta_{1\ldots n \ovl{1}\dots \ovl{n}}}
\right)
\;dz^1\wedge \ldots \wedge dz^n \wedge dz^{\ovl{1}}\wedge \ldots \wedge dz^{\ovl{n}}.
$$
Here $;$ denotes covariant derivative. Now because of
$$
a_k^{\la} = O(|z^\la| \log |z^{\la}|) \quad \mbox{ for } \quad 1\leq \la \leq k
$$
we see that $L_{v_k}(\eta)$ is indeed integrable.
\end{proof}
We see that we can iterate this process so that the fiber integral gives a smooth function on $S$. But this means that the $L^2$-metric is indeed a smooth metric on $f_*{(\Omega^n(\log \D)_{\X/S}(\LL)})$.
We note here that the square integrability of $L_v(\eta)$ will also follow from Lemma \ref{integrability}.

Before we go to the computation of the curvature which is the most technical part of the article, let us consider the two applications.
\section{Families of log-canonically polarized manifolds: (Semi-)Positivity of the relative canonical bundle}
Let $\DD \overset{i}{\hookrightarrow} \X \overset{f}{\rightarrow} S$ be a holomorphic \emph{family of log-canonically polarized pairs}, i.e. holomorphic family of smooth log pairs $(X_s,D_s)$ with ample adjoint bundle $K_{X_s} + D_s$. The family is called {\em effectively parameterized} if the \ks\ map
$$
\rho_s: T_sS \to H^1_{(2)}(X'_s,T_{X'_s})
$$
is injective at all points $s\in S$.

Let $\{\omega_s\}_{s \in S}$ be the family of complete K\"ahler-Einstein metrics on $X_s'=X_s \setminus D_s$ with constant negative curvature $-1$ given by \cite{Kob84,TY86}.
By \cite[Eq.(1)]{Sch98} they fulfill the Monge-Amp\`ere equation
\begin{equation}
\label{MA}
\omega_s^n = \exp(u_s) \cdot (\psi |X_s),
\end{equation}
where $\psi=\Omega/(\prod_i{||\sigma_i||^2 \log^2||\sigma_i||^2})$ for a smooth relative volume form $\Omega$. Here $u_s$ is a family of functions in $\Cka^{k,\la}(X\setminus D)$. 
By the implicit function theorem applied to the H\"older spaces $\Cka^{k,\la}(X\setminus D)$ of functions, these functions $u_s$ depend smoothly on the parameter $s\in S$ in the sense that the map 
$s \mapsto \Cka^{k,\la}$ is indeed Fr\'echet differentiable (cf. \cite[Sect. 2.6]{Sch98}). 
Hence we can consider the relative volume form $\omega^n_{\X'/S}$ on $\X'=\X\setminus \DD$ associated to the family $\{\omega_s\}_{s \in S}$ as a singular hermitian metric on $K_{\X/S}$ whose inverse has Poincar\'e type singularities along $\D$. Its smooth curvature form on $\X'$ is given by
$$
\omega_{\X'}:= -\Ric(\omega^n_{\X'/S}).
$$
Analogous to the canonically polarized case proved in \cite{Sch12}, we have the following result: 
\begin{theorem}
\label{posvar}
The form $\omega_{\X'} \geq 0$ is semi-positive, and strictly positive if the family $\DD \overset{i}{\hookrightarrow} \X \overset{f}{\rightarrow} S$ is effectively parameterized.
\end{theorem}
\begin{proof}
The computation from \cite{Sch98} can be adopted. We summarize the main points that we need.

Given a coordinate vector field  $\pt /\pt s^i $ on $S$, and the horizontal lift $v_i$, define
$$
\varphi_{i \ol\imath} = \langle v_i, v_i  \rangle_{\omega_{\X'}}
$$
Then
\begin{equation}\label{eq:phi}
(1 + \Box_s) \varphi_{i \ol\imath} = \|A_i\|^2(z,s),
\end{equation}
where $\Box_s$ denotes the (semi-positive) Laplacian, and $\|A_i\|(z,s)$ the pointwise norm of the harmonic representative of $\rho_s(\pt/\pt s^i)$.

The results of the previous section imply that the quantities occurring in  \eqref{eq:phi} are $\Cka^{k,\la}$-tensors on the total space and also define such tensors, when restricted to the fibers of $f$, and of class $C^\infty$ on $\X'$. Yau's maximum principle \cite[Theorem 1]{Yau78} applies to restrictions of \eqref{eq:phi} to the fibers of $f$ immediately yields that $\varphi_{i\ol\imath}\geq 0$. 

The integral of \eqref{eq:phi} along a fiber yields the \wp\ norm of $\pt/\pt s^i|_s$:
$$
\|\pt/\pt s^i|_s\|^2_{WP} = \int_{X_s}\|A_i\|^2(z,s) g\/ dV. 
$$
(Again we are using Gaffney's result). 

Let $\rho_s(\pt/\pt s^i)\neq 0$, i.e.\ $A_i\neq0$. one can show that $\varphi_{i\ol\imath}(z,s)$ has no zeroes. This follows from the lower heat kernel estimate in the complete case  as given in \cite[Cor. 4.3]{St92}. This shows that the heat kernel is strictly positive on the fiber $X_s'$. Then the argument is the same as in \cite[Prop. 1]{Sch12} except that we do not have a fixed positive lower bound in terms of the diameter of the fibers.  
\end{proof}

\begin{corollary}
\label{pshvar}
For a family of log-canonically polarized manifolds $\DD \overset{i}{\hookrightarrow} \X \overset{f}{\rightarrow} S$ the relative adjoint bundle $K_{\X/S} \otimes \DD$ is nef. If the family is effectively parameterized,
$K_{\X/S} \otimes \DD$ is big.
\end{corollary}
\begin{proof}
We compute the curvature current of the singular hermitian metric $(\omega_{\X'/S}^n)^{-1}=(\omega_s^n)^{-1}_{s \in S}$ on $K_{\X/S}$ now on the whole $\X$. From the  Monge-Amp\`ere equation (\ref{MA}) we see that the only additional term that vanishes by restricting to $\X'$ is just $-\DD$ which comes from the term $-\sum \dl \dbar \log ||\sigma_i||^2$; it is compensated by adding $\DD$. This shows that 
$K_{\X/S} \otimes \DD$ is pseudoeffective. Here we take the canonical singular metric on $\DD$. The nefness follows from the fact that the curvature current of the metric on $K_{\X/S} \otimes \DD$ has zero Lelong numbers. The bigness in the effectively parameterized situation then follows from the strict positivity of $K_{\X'/S}$ and Boucksom's bigness criterium \cite[Cor. 3.3]{Bou02}. 
\end{proof}

\section{The case of a big line bundle}
Let $E \overset{i}{\hookrightarrow} \X \overset{f}{\rightarrow} S$ be a holomorphic family of smooth log pairs $(X_s,E_s)$ and $F$ a big line bundle on $\X$. We assume that we have a decomposition
$$
F= A + E
$$
where $A$ is ample on $\X$. We choose a smooth positive metric $h_A$ on $A$ and a smooth hermitian metric  $h_i$ on each irreducible component $E_i$ of $E$. Using the canonical section $\sigma_i$ cutting out the divisor $E_i$, we define another metric on $A$ by setting
$$
h_{A,\varepsilon} := h_A \cdot \left(\prod_i{|\sigma_i|_{h_i}^2\log^2(|\sigma_i|_{h_i}^2)} \right)^{\varepsilon}
$$
We now make the assumption that there exists an $\varepsilon >0$ such that 
$$
\mathrm{i} \,\Theta_{h_{A,\varepsilon}}(A) =   \mathrm{i} \,\Theta_{h_A}(A) -  \varepsilon \; \sqrt{-1}\sum_i{ \dl \dbar \log\left(|\sigma_i|_{h_i}^2 \log^2(|\sigma_i|_{h_i}^2)\right)} >0 \quad \mbox{on} \quad \X'.
$$
In general, such an $\varepsilon$ need not to exist and because it depends on the curvatures of $h_A$ and $h_i$ on $\X$.
Then we can equip $F$ with the metric $h_{F,\varepsilon}$ defined by
$$
h_{F,\varepsilon}= h_{A,\varepsilon} \cdot h_{E,\operatorname{sing}},
$$
where $h_{E,\operatorname{sing}}$ is the canonical singular hermitian metric on $E$ given by the section $\sigma=\prod_i{\sigma_i}$. We note that $(E,h_{E,\operatorname{sing}})|_{\X'}$ is the trivial line bundle equipped with the trivial metric. 

Now the hermitian bundle $(A,h_{A,\varepsilon})$ fulfils the requirements of our main theorem (note here that it works with a power $\varepsilon>0$ instead of $1$ as well) and we get the Nakano positivity of
$$
f_*(K_{\X/S}+E+A)=f_*(K_{\X/S}+F).
$$
If we apply instead the general result from \cite{BP08} directly to the hermitian bundle $(F,h_F)$ with $h_F=h_A \cdot h_{E,\operatorname{sing}}$ we get first that
$$
f_*((K_{\X/S}+F)\otimes \mathcal{J}(h_F))
$$
is positive in the singular sense of Griffiths. But here we have $ \mathcal{J}(h_F)=\Oh(-E)$, so we can conclude that
$$
f_*((K_{\X/S}+F)\otimes \mathcal{J}(h_F)) = f_*(K_{\X/S}+A)
$$
is in fact positive in the sense of Nakano using the result from \cite{Be09}. Of course we should mention here that we have
$$
f_*(K_{\X/S}+A) \subset f_*(K_{\X/S}+E+A)
$$
as sheaves but not as hermitian bundles because we changed the metric on $A$ for the larger one. If the decomposition $F=A+E$ is a relative Zariski decomposition, both sheaves coincide. 
If one applies \cite{BP08} to $(L,h_{F,\varepsilon})$ with a trivial multiplier ideal sheaf $\mathcal{J}(h_{F,\varepsilon})$ we get Griffiths positivity of $f_*(K_{\X/S}+F)$ only in the weaker singular sense. 

\section{Computation of the curvature}
Computing the curvature of the $L^2$-metric on $f_*(\Omega^n(\log \DD))_{\X/S}(\LL)$ requires taking derivatives in the base direction of fiber integrals, which can be realized by taking Lie derivatives of the integrands. These Lie derivatives can be split up by introducing Lie derivatives of $(n,0)$-forms with values in $L$. They are computed in terms of covariant derivatives with respect to the Chern connection on
$(X'_s,\omega_s)$ and the hermitian holomorphic bundle $(L_s,h_s)$.
We use the symbol $;$ for covariant derivatives and $,$ for ordinary derivatives. Greek letters indicate the fiber direction, whereas Latin indices stand for directions on the base. Because we are dealing with alternating $(p,q)$-forms, the coefficients are meant to be skew-symmetric. Thus every such $(p,q)$-form carries a factor $1/p!q!$, which we suppress in the notation. These factors play a role in the process of skew-symmetrizing the coefficients of a $(p,q)$-form by taking alternating sums of the (not yet skew-symmetric) coefficients. We adopt the Einstein convention of summation.

\subsection{Setup}\label{setup}
By polarization, it is sufficient to treat the case where $\dim S=1$ for the computation of the curvature, which simplifies the notation. Therefore, we set $s=s^1, v_s=v_1$, etc. We write $s,\sbar$ for the indices $1,\ovl{1}$ so that
$$
v_s = \dl_s + a_{s}^{\la}\dl_{\la}
$$
and
$$
A_s = A_{s\lb}^{\la}\dl_{\la}dz^{\lb}.
$$
We assume local freeness of the sheaf $f_*(\Omega^n(\log \DD))_{\X/S}(\LL)$. According to Corollary \ref{cor:repr}, we can represent local sections of this sheaf by holomorphic 
sections of $(\Omega^n(\log \D))_{\X'/S}(\LL|_{\X'})$, which restrict to holomorphic and square integrable $(n,0)$-forms on the open fibers $X'_s$. We denote such a section by $\psi$. In local coordinates, we have
\begin{eqnarray*}
\psi|_{X'_s} &=& \psi_{\la_1 \ldots \la_n}dz^{\la_1}\we \ldots dz^{\la_n}\\
&=& \psi_{A_n}dz^{A_n},
\end{eqnarray*}
where $A_n=(\la_1,\ldots,\la_n)$. The $\dbar$-closedness of $\psi$ means that
\begin{eqnarray}
\label{dbar-closedness}
\psi_{A_n;s} = 0 \quad \mbox{and} \quad
\psi_{A_n;\lb} = 0 \quad \mbox{for all} \quad A_n, 1\leq \beta \leq n.
\end{eqnarray}

\subsection{Cup product}
\begin{definition}
Let $s \in S$ and $A=A_{s\lb}^{\la}(z,s)\dl_{\la}dz^{\lb}$ be the Kodaira-Spencer form on the fiber $X'_s$. The wedge product together with the contraction defines a map 
$$
A_{i\lb}^{\la}\dl_{\la}dz^{\lb}\cup \; : H^0(X_s,\Omega^n_{X_s}(\log D_s)(L|_{X_s})) \to A^{0,1}_{(2)}(X'_s,\Omega^{n-1}_{X'_s}(L|_{X'_s})), 
$$
which can be described locally by
\begin{eqnarray*}
&&\left( A_{i\ld}^{\lc}\dl_{\lc}dz^{\ld}\right) \cup \left(\psi_{\la_1 \ldots \la_n}\;dz^{\la_1}\we \ldots \we dz^{\la_n} \right)\\
&=& A_{i\lb}^{\lc} \psi_{\lc\la_1\ldots \la_{n-1}}\; dz^{\lb}\we dz^{\la_1} \we \ldots \we dz^{\la_{n-1}}.
\end{eqnarray*}
\end{definition}
The fact that $A_i \cup \psi$ is indeed square integrable will be proved in Lemma \ref{integrability}.

\subsection{Lie derivatives}
Now we choose a local frame $\{\psi^1,\ldots,\psi^r\}$ according to Corollary \ref{cor:repr}.
The components of the metric tensor $H^{\lbar k}$ for $f_*(\Omega^n(\log \DD))_{\X/S}(\LL)$ on the base space $S$ are given by 
$$
H^{\lbar k}(s):=\<\psi^k,\psi^l\> :=\<\psi^k|_{X'_s},\psi^l|_{X'_s}\> = \int_{X'_s}{\psi^k_{A_n}\psi^{\lbar}_{\ovl{B}_n}g^{\ovl{B}_nA_n} h|_{X'_s}\,dV}.
$$
We also write
$$
\psi^k \cdot \psi^{\lbar}= \psi^k_{A_n}\psi^{\lbar}_{\ovl{B}_n}g^{\ovl{B}_nA_n} h|_{X'_s}
$$
for the pointwise inner product of $L|_{X'_s}$-valued $(n,0)$-forms. Here and in the following we write $g$ for the hermitian metric associated to the complete K\"ahler from $\omega_s$.
When we compute derivatives with respect to the base of these fiber integrals, we apply Lie derivatives with respect to the horizontal lifts of the tangent vectors according to Lemma \ref{lemma 1}. This simplifies the computation in a considerable way. In order to break up the Lie derivative of the pointwise inner product (which is a relative $(n,n)$-form), we need to introduce Lie derivatives of relative differential forms with values in a line bundle. This can be done by using the hermitian connection $\nabla$ on 
$\Lambda^{n,0}T^*_{\X'/S}\otimes \LL|_{\X'}$ induced by the Chern connections on $(T_{X'_s},\omega_{X_s})$ and $(L_s,h_s)$. We define the Lie derivative of $\psi$ with respect to the horizontal lift $v$ by using Cartan's formula
\begin{equation}
\label{DefLieDer}
L_v\psi := L_v (\psi_{\X'/S}) := \left(\delta_v\circ \nabla + \nabla \circ \delta_v\right)\psi
\end{equation}
and similar for the Lie derivative with respect to $\vbar$. 

Taking Lie derivatives is not type-preserving. We have the type decomposition for $\psi=\psi^k$ or $\psi=\psi^l$ and $v=v_s$
$$
L_v\psi = L_v\psi' + L_v\psi'',
$$
where $L_v\psi'$ is of type $(n,0)$ and $L_v\psi''$ is of type $(n-1,1)$. In local coordinates, we have
\begin{eqnarray}
\label{Lv'}
L_v\psi'  = \left(\psi_{A_n;s} + a_s^{\la}\psi_{A_n;\la} + \sum_{j=1}^n{a_{s;\la_j}^{\la}
\psi_{
{\tiny\vtop{
\hbox{$\la_1\ldots\la\ldots\la_n$}\vskip-.8mm
\hbox{$\phantom{\la_1\ldots}{|\atop j} $}}}}}
\right)
\;dz^{A_n}
\end{eqnarray}
\begin{eqnarray}
\label{Lv''}
L_v\psi'' &=& \sum^n_{j=1} A^\la_{s\lb_n}
\psi_{ {\tiny\vtop{ \hbox{$\la_1\ldots\la\ldots\la_n$\;}\vskip-.8mm \hbox{$\phantom{\la_1\ldots}{|\atop j} $}}}}
\vtop{\hbox{$dz^{\la_1}\wedge\ldots\wedge dz^{\lb_n}\wedge\ldots\wedge
dz^{\la_n} 
$}\hbox{$\phantom{dz^{\la_1}\we\ldots\we \; }{|\atop j} $}}
\end{eqnarray}
One justification for using Lie derivatives is given by the following lemma, which allows us to express some components of the Lie derivatives as cup products with the Kodaira-Spencer form:
\begin{lemma} \label{prim}
We have
\begin{eqnarray}
L_v\psi''&=&A_s \cup \psi\label{id1}
\end{eqnarray}
and it is primitive on the fibers.
\end{lemma}
\begin{proof}
First we note
\begin{eqnarray*}
L_v\psi'' &=& \sum^n_{j=1} A^\la_{s\lb_n}
\psi_{ {\tiny\vtop{ \hbox{$\la_1\ldots\la\ldots\la_n$\;}\vskip-.8mm \hbox{$\phantom{\la_1\ldots}{|\atop j} $}}}}
\vtop{\hbox{$dz^{\la_1}\wedge\ldots\wedge dz^{\lb_n}\wedge\ldots\wedge
dz^{\la_n} 
$}\hbox{$\phantom{dz^{\la_1}\we\ldots\we \; }{|\atop j} $}}\\
&=& \sum^n_{j=1}
A^\la_{s\lb_n}\psi_{\la\,\la_1\ldots\la_{n-1}} dz^{\lb_n}\we dz^{\la_1}\we
\ldots \we dz^{\la_{n-1}}.
\end{eqnarray*}
To prove that $A_s \cup \psi$ is primitive, we have to show that $\Lambda_s (A_s\cup \psi)=0$ where $\Lambda_s$ is the dual Lefschetz operator with respect to the K\"ahler form
$$
\omega_s = \sqrt{-1}g_{\la\lb}\;dz^{\la}\wedge dz^{\lb}.
$$
We have
\begin{eqnarray*}
(\Lambda_s (A_s \cup \psi))_{\la_2 \ldots \la_{n-1}} &=& g^{\lb_n \la_1} A^\la_{s\lb_n}\psi_{\la\,\la_1\ldots\la_{n-1}} = A^{\la\la_1}_s\psi_{\la\,\la_1\ldots\la_{n-1}}.
\end{eqnarray*}
But now because $A^{\la\la_1}_s=A^{\la_1\la}_s$ by Lemma \ref{symm} and $\psi_{\la\,\la_1\ldots\la_{n-1}}$ is skew-symmetric we get that
$$
(\Lambda_s (A_s \cup \psi))_{\la_2 \ldots \la_{n-1}} dz^{\la_2} \wedge \ldots \wedge dz^{\la_{n-1}} =0.
$$

\end{proof}
Similarly we have a type decomposition for the Lie derivative along $\vbar = v_{\sbar}$
$$
L_{\vbar}\psi = L_{\vbar}\psi' + L_{\vbar}\psi'',
$$
where $L_{\vbar}\psi'$ is of type $(n,0)$ and $L_{\vbar}\psi''$ is of type $(n+1,n-1)$ and hence vanishes by degree reasons. In local coordinates, this is
\begin{eqnarray}
\label{Lvbar'}
L_{\vbar}\psi'=\left(\psi_{A_p\Bbar_{n-p};\sbar}+a_{\sbar}^{\lb}\psi_{A_p\Bbar_{n-p};\lb} + \sum_{j=p+1}^n{a_{\sbar;\lb_j}^{\lb}\psi_{
{\tiny\vtop{
\hbox{$A_p \lb_{p+1}\ldots\lb\ldots\lb_n$}\vskip-.8mm
\hbox{$\phantom{A_p \cbar_{p+1}\ldots}{|\atop j}$}}}}}
\right)
\;dz^{A_p}\wedge dz^{\Bbar_{n-p}}.
\end{eqnarray}
Form this we infer that $L_{\vbar}\psi=L_{\vbar}\psi'=0$ because $\psi$ is holomorphic.
The type decomposition can be verified using definition ($\ref{DefLieDer}$). We refer the reader to \cite{Na17} for a verification. 

\begin{lemma}\label{integrability}
The smooth forms $L_v \psi'$ and $L_v\psi''$ are $L^2$-integrable.
\end{lemma}
\begin{proof}
We use local coordinates $z^1,\ldots,z^n$ in a neighbourhood $U$ of a point $p \in D_s \subset X_s$ where $D_s \cap U = \{z^1\cdots z^k=0\}$.
First we now that we have
$$
\psi_{A_n} = O\left(\frac{1}{|z^1\cdots z^k|}\right)
$$
and
$$
a_s^{\la} = O(|z^{\la}| \log |z^{\la}|) \quad \mbox{for} \quad 1 \leq \la \leq k \quad \mbox{else} \quad a_s^{\la} = O(1).
$$
Form the local expression (\ref{Lv'}) we thus see that 
$$
(L_v\psi)'_{A_n} = O\left(\frac{1}{|z^1\cdots z^k|}\right),
$$
so it is again square integrable. 

To prove that $L_v\psi''$ is square integrable is more complicated. We first look at the order of $A^{\la}_{s\lb}$:
\begin{eqnarray*}
A^{\la}_{\lb} &=& O(1) \quad \mbox{for} \quad 1\leq \la=\beta \leq k \quad \mbox{or } \quad k+1 \leq \la,\beta \leq n\\
A^{\la}_{\lb} &=& O\left(\frac{1}{|z^{\beta}|\log |z^{\beta}|}\right) \quad \mbox{for} \quad \la>k \mbox{ and } \beta\leq k \\
A^{\la}_{\lb} &=& O\left(|z^{\la}|\log |z^{\la}|\right) \quad \mbox{for} \quad \la \leq k \mbox{ and } \beta > k \\
A^{\la}_{\lb} &=& O\left(\frac{|z^{\la}|\log |z^{\la}|}{|z^{\beta}|\log |z^{\beta}|}\right) \quad \mbox{for} \quad 1 \leq \la\neq \beta \leq k. 
\end{eqnarray*}
To prove that $L_v\psi''=A_s \cup \psi$ is $L^2$-integrable means to verify that
$$
\int_{X'_s}{(A_s \cup \psi) \wedge \ovl{(A_s \cup\psi)}}
$$
is finite because the form is primitive by Lemma \ref{prim}. For this we first note that sum in the expression of $(A_s \cup \psi)_{\lb_n\la_1\ldots \la_{n-1}}$ reduces to
$$
A^{\la_n}_{s\;\lb_n}\psi_{\la_n\la_1\ldots\la_{n-1}}.
$$

The only really critical term in $(A_s \cup \psi)_{\lb_n\la_1\ldots \la_{n-1}}$ occurs if $\la_n \in \{k+1,\dots,n\}$, 
$\beta_n \in \{1,\ldots,k\}$ and $\beta_n$ is among the $\la_1,\ldots, \la_{n-1}$ because in this case the order in $z^{\beta_n}$ is
$$
\frac{1}{|z^{\beta_n}|^2 \log |z^{\beta_n}|}.
$$
But then in the above integral this term can only be paired with another term that does neither contain $dz^{\beta_n}$ nor $dz^{\lb_n}$. So wee see that the product
$(A_s \cup \psi) \wedge \ovl{(A_s \cup\psi)}$ remains integrable.
 \end{proof}

We need the following lemma
\begin{lemma} \label{le:dV}
The Lie derivative of the volume element $dV=\omega_s^n/n!$ along the horizontal lift $v$ vanishes, i.e.
$$
L_v( dV)=0.
$$
\end{lemma}
\begin{proof}
It suffices  to show that the $(1,1)$ component of $L_v(g_{\la\lb})$ vanishes, which implies $L_v(\det(g_{\la\lb}))=0.$ We have
$$
L_v(g_{\la\lb})_{\la\lb} = g_{\la\lb,s} + a_{s}^{\lc}g_{\la\lb;\lc} + a_{s;\la}^{\lc}g_{\lc\lb} = -a_{s\lb;\la} + a_{s;\la}^{\lc}g_{\lc\lb}=0.
$$
\end{proof}

\subsection{Main part of the computation}
We start computing the curvature by computing the first order variation. Using Lie derivatives, the pointwise inner products can be broken up:
\begin{proposition}
\label{firstordervar}
$$
\frac{\dl}{\dl s} \<\psi^k,\psi^l\> = \<L_v\psi^k,\psi^l\>,
$$
where $\dl/\dl s$ denotes a tangent vector on the base $S$ and $v$ its horizontal lift and analogous for $\dl/\dl \sbar$.
\end{proposition}
\begin{proof}
We first apply Lemma \ref{lemma 1} and get that 
$$ 
\frac{\dl}{\dl s} \<\psi^k,\psi^l\> (s) = \int_{X'_s}{L_v(\psi^k \cdot \psi^{\lbar}\; dV)} = \int_{X'_s}{L_v(\psi^k \cdot \psi^{\lbar})\; dV}
$$
by Lemma \ref{le:dV}.  Now it follows by a direct computation (see \cite[Prop.1]{Na17}) that 
$$
L_v(\psi^k \cdot \psi^{\lbar}) = L_v\psi^k \cdot \psi^{\lbar} + \psi^k \cdot L_v\psi^{\lbar}
$$
so that
$$
\frac{\dl}{\dl s} \<\psi^k,\psi^l\> = \<L_v\psi^k,\psi^l\> +  \<\psi^k,L_{\vbar}\psi^l\> = \<L_v\psi^k,\psi^l\>
$$
because $L_{\vbar}\psi^l=0$.
\end{proof}

The above proposition is a main reason for the use of Lie derivatives. For later computations, we need to compare Laplacians:
\begin{lemma}
\label{BKN}
We have the following relation on the space $A^{p,q}_{(2)}(X'_s,L_s)$:
\begin{equation}
\laplace - \laplacedbar = (n-p-q)\cdot \id
\end{equation}
In particular, the harmonic forms $\psi \in A^{n,0}_{(2)}(X'_s,L_s)$ are also harmonic with respect to $\dl$, which is the $(1,0)$- part of the hermitian connection on 
$A^{n,0}_{(2)}(X'_s,L_s)$.
\end{lemma}
\begin{proof}
The Bochner-Kodaira-Nakano identity says (on the fiber $X'_s$)
$$
\laplacedbar - \laplace = \left[\sqrt{-1}\Theta(L_s),\Lambda\right].
$$
But by definition, we have $\omega_{X_s}=\sqrt{-1}\Theta(L_s)$. Furthermore, it holds (see \cite[Cor.VI.5.9]{De12})
$$
\left[L_{\omega},\Lambda_{\omega}\right]u = (p+q-n)\,u \quad \mbox{for } u \in \A^{p,q}(X'_s,L_s).
$$
\end{proof}
Next, we start to compute the second order derivative of $H^{\lbar k}$ and begin with
$$
\frac{\dl}{\dl s} H^{\lbar k} = \<L_v \psi^k,\psi^l\>.
$$
We obtain
\begin{eqnarray*}
\dl_{\sbar}\dl_s \<\psi^k,\psi^l\> &=& \<L_{\vbar}L_v\psi^k,\psi^l\> + \<L_v\psi^k,L_v\psi^l\>\\
&=& \<(L_{[\vbar,v]} + \Theta(L|_{\X'})_{\vbar v})\psi^k, \psi^l\> +  \<L_vL_{\vbar}\psi^k,\psi^l\> + \<L_v\psi^k,L_v\psi^l\>\\
&=& \<(L_{[\vbar,v]} + \Theta(L|_{\X'})_{\vbar v})\psi^k, \psi^l\> + \dl_s\<L_{\vbar}\psi^k,\psi^l\> - \<L_{\vbar}\psi^k,L_{\vbar}\psi^l\> + \<L_v\psi^k,L_v\psi^l\>.
\end{eqnarray*}
Because of $L_{\vbar}\psi^k \equiv 0$ as we just saw we get
\begin{equation}
\label{secondorder}
\dl_{\sbar}\dl_s \<\psi^k,\psi^l\> = \<(L_{[\vbar,v]} + \Theta(L|_{\X'})_{\vbar v})\psi^k,\psi^l\> + \<L_v\psi^k,L_v\psi^l\> .
\end{equation}
We will see below that the smooth $(n,0)$-form $(L_{[\vbar,v]} + \Theta(L|_{\X'})_{\vbar v})\psi^k$ is indeed square integrable which justifies that $L_{\vbar}L_v \psi^k$ is square integrable, too.

Now we treat each term on the right hand side of (\ref{secondorder}) separately. For the first summand, we have
\begin{lemma}
\begin{equation}
L_{[\vbar,v]} + \Theta(L_{\X'})_{\vbar v}= [-\varphi^{;\la}\dl_{\la} + \varphi^{;\lb}\dl_{\lb},\rule{0.3cm}{0.4pt}] - \varphi \cdot \id,
\end{equation}
where the bracket $[w,\rule{0.3cm}{0.4pt}]$ stands for a Lie derivative along the vector field $w$.
\end{lemma}
\begin{proof}
We first compute the vector field $[\vbar,v]$:
\begin{eqnarray*}
[\vbar,v] &=& [\dl_{\sbar} + a_{\sbar}^{\lb}\dl_{\sbar},\dl_s + a_s^{\la}\dl_{\la}] \\
&=& \left(\dl_{\sbar}(a_s^{\la}) + a_{\sbar}^{\lb}a_{a|\lb}^{\la} \right)\dl_{\la} - \left(\dl_s(a_{\sbar}^{\lb}) + a_s^{\la}a_{\sbar|\la}^{\lb} \right)\dl_{\lb}
\end{eqnarray*}
Now we have
\begin{eqnarray*}
\dl_{\sbar}(a_s^{\la}) &=& -\dl_{\sbar}(g^{\lb\la}g_{s\lb}) = g^{\lb\ls}g_{\ls\sbar|\ovl{\lt}}g^{\ovl{\lt}\la}g_{s\lb} - g^{\lb\la}g_{s\lb|s}\\
&=& g^{\lb\ls}a_{\sbar\ls;\ovl{\lt}}g^{\ovl{\lt}\la}a_{s\lb} - g^{\lb\la}g_{s\sbar;\lb}
\end{eqnarray*}
Because of $\varphi = g_{s\sbar} - g_{\la \sbar}g_{s\lb}g^{\lb\la}$ the coefficient of $\dl_{\la}$ is $g^{\lb\la}\varphi_{;\lb}=\varphi^{;\la}$. In the same way we get the coefficient of $\dl_{\lb}$.
Next, we need to compute the contribution of the connection on $L|_{\X'}$.
Because of $\sqrt{-1}[\dl,\dbar]=\sqrt{-1}\Theta(L)|_{\X'}=\omega_{\X'}$, we have
\begin{eqnarray*}
\Theta(L|_{\X'})_{\vbar v} &=& -\Theta(L|_{\X'})_{v \vbar}\\
&=& -\left(g_{s\sbar} + a_{\sbar}^{\lb}g_{s\lb} + a_s^{\la} g_{\la \sbar} + a_{\sbar}^{\lb}a_s^{\la}g_{\la\lb}\right)\\
&=& -\varphi.
\end{eqnarray*}
\end{proof}

\begin{lemma}
\begin{equation}
\label{term1}
\<(L_{[\vbar,v]}+ \Theta(L|_{\X'})_{\vbar v})\psi^k,\psi^l\> = -\<\varphi \cdot\psi^k,\psi^l\> = -\int_{X'_s}{\varphi\cdot(\psi^k\cdot\psi^{\lbar})\,dV}.
\end{equation}
\end{lemma}
\begin{proof}
The $\dl$-closedness of $\psi^k$ means that
$$
\psi^k_{;\la} = \sum_{j=1}^n
\psi^k_{
{\tiny\vtop{
\hbox{$\la_1 \ldots \la \ldots \la_n;\la_j\;$}\vskip-.8mm
\hbox{$\phantom{\la_1\ldots}{|\atop j} $}}}}.
$$
Thus
\begin{eqnarray*}
[\varphi^{;\la}\dl_{\la},\psi^k_{A_n}]' &=& \varphi^{;\la}\psi^k_{;\la} +
\sum_{j=1}^n
\varphi^{;\la}_{;\la_j}\psi^k_{
{\tiny\vtop{
\hbox{$\la_1 \ldots \la \ldots \la_n\;$}\vskip-.8mm
\hbox{$\phantom{\la_1\ldots}{|\atop j} $}}}}\\
&=& \sum_{j=1}^n (
\varphi^{;\la}\psi^k_{
{\tiny\vtop{
\hbox{$\la_1 \ldots \la \ldots \la_n\;$}\vskip-.8mm
\hbox{$\phantom{\la_1\ldots}{|\atop j} $}}}}
)_{;\la_j}\\
&=& \dl \left( \varphi^{;\la}\dl_{\la}\cup \psi^k\right).
\end{eqnarray*}
It is clear that $(\varphi^{;\la} \dl_{\la}\cup \psi^k)$ is square integrable, because $\varphi|_{X'_s}$ lies in $\Cka^{k,\la}(X'_s)$. Moreover it guarantees that this form lies in the domain of $\dl$.
This leads to
\begin{eqnarray*}
\<[\varphi^{;\la}\dl_{\la},\psi^k_{A_n}],\psi^l\> &=& \<[\varphi^{;\la}\dl_{\la},\psi^k_{A_n}]',\psi^l\>\\
&=& \<\dl\left(\varphi^{;\la}\dl_{\la}\cup \psi^k\right),\psi^l\> = \<\varphi^{;\la}\dl_{\la}\cup \psi^k, \dl^*\psi^l\> =0.
\end{eqnarray*}
Note that by Gaffney's theorem, Proposition \ref{GaffneyStokes}, the formal adjoint of $\dl$ is equal to the adjoint operator.

In the same way we get
$$
\<[\varphi^{;\lb}\dl_{\lb},\psi^k_{A_n}],\psi^l\> =0.
$$
\end{proof}

The following proposition contains important identities that allow to obtain an intrinsic expression for the curvature:
\begin{proposition}
\label{basicid}
\begin{eqnarray}
\dbar(L_v\psi^k)'&=&\dl(A_s\cup\psi^k), \label{eq:1}\\
\dbar^*(L_v\psi^k)'&=& 0, \label{eq:2}\\
\dl^*(A_s \cup \psi^k)&=&0. \label{eq:3}
\end{eqnarray}
\end{proposition}
We note that here the operators $\dl,\dbar, \dl^*$ and $\dbar^*$ mean the fiberwise operators, because we are always dealing with relative forms.
For a proof we refer to \cite[Appendix A]{Na17}.
We see form the proof of Lemma \ref{integrability} that $\dbar(L_v \psi^k)'$ is again square integrable. 

Now we look at the second term in (\ref{secondorder}) and decompose it into its two types:
\begin{eqnarray*}
\<L_v\psi^k,L_v\psi^l\> &=& \<(L_v\psi^k)',(L_v\psi^l)'\> - \<(L_v\psi^k)'',(L_v\psi^l)''\>\\
&=&  \<(L_v\psi^k)',(L_v\psi^l)'\> - \<A_s \cup \psi^k, A_s \cup \psi^l\>
\end{eqnarray*}
because of (\ref{id1}).

Now let $G_{\dl}$ and $G_{\dbar}$ be the Green operators on the spaces $A^{p,q}_{(2)}(X'_s,L|_{X'_s})$ with respect to $\laplace$ and $\laplacedbar$ respectively. According to Lemma \ref{BKN} they coincide for $p+q=n$. Now we use normal coordinates (of the second kind) at a given point $s_0 \in S$. The condition $(\dl/\dl s)H^{\lbar k}|_{s_0}=0$ for all $k,l$ means that for $s=s_0$ the harmonic projection
$$
H((L_v\psi^k)')=0
$$
vanishes for all $k$. Thus, using the identity $\id = H + G_{\dbar}\laplacedbar$ we can write
$$
(L_v\psi^k)'=G_{\dbar}\laplacedbar(L_v\psi^k)' = G_{\dbar}\dbar^*\dbar(L_v\psi^k)' = \dbar^*G_{\dbar}\dl(A_s\cup \psi^k)
$$
by (\ref{eq:2}) and (\ref{eq:1}). Because the form $\dbar(L_v\psi^k)'=\dl(A_s \cup \psi^k)$ is of type $(n,1)$, we have
$G_{\dbar}=(\laplace + 1)^{-1}$ on such forms by Lemma \ref{BKN}. We proceed by
\begin{eqnarray*}
\<(L_v\psi^k)',(L_v\psi^l)'\> &=& \<\dbar^*G_{\dbar}\dl(A_s \cup \psi^k),(L_v\psi^l)'\>\\
&=&  \<G_{\dbar}\dl(A_s \cup \psi^k),\dl(A_s \cup \psi^l)\>\\
&=&  \<(\laplace + 1)^{-1}\dl(A_s \cup \psi^k),\dl(A_s \cup \psi^l)\>\\
&=&  \<\dl^*(\laplace + 1)^{-1}\dl(A_s \cup \psi^k),A_s \cup \psi^l\>.
\end{eqnarray*}
Again we used Gaffney's theorem.
Now using (\ref{eq:3}) gives
\begin{eqnarray*}
\<(L_v\psi^k)',(L_v\psi^l)'\> &=& \<(\laplace + 1)^{-1}\laplace(A_s \cup \psi^k),A_s \cup \psi^l\>\\
&=& \<(\laplace + 1)^{-1}(\laplace + 1 - 1) (A_s \cup \psi^k),A_s \cup \psi^l\>\\
&=& \<A_s \cup \psi^k,A_s \cup \psi^l\> - \<(\laplace + 1)^{-1}(A_s \cup \psi^k),A_s \cup \psi^l\>.
\end{eqnarray*}
Altogether, we have
\begin{lemma}
\begin{equation}
\label{term2}
\<L_v\psi^k,L_v\psi^l\> = - \int_{X'_s}{(\Box + 1)^{-1}(A_s \cup \psi^k)\cdot (A_{\sbar} \cup \psi^{\lbar})\, g\, dV}
\end{equation}
(We write $\Box=\laplace=\laplacedbar$ when applied to $(n-1,1)$-forms.)
\end{lemma}

Now our main result Theorem \ref{mainresult} follows form (\ref{secondorder}), (\ref{term1}), (\ref{term2}) and the fact that
$R^{\lbar k}_{i\jbar}(s_0)=-\dl_{\jbar}\dl_i H^{\lbar k}(s_0)$ in normal coordinates at a point $s_0 \in S$.

\end{document}